\documentclass[9pt,twoside,a4paper]{article}
\usepackage{amssymb, amsmath, amsfonts, amsthm,mathrsfs}
\usepackage{graphics}
\usepackage[dvips]{graphicx}
\usepackage{eepic}
\usepackage{epic}
\usepackage{tikz}
\usepackage{layout}

\newcommand{\mbb}[1]{\mathbb{#1}}
\newcommand{\mbs}[1]{\boldsymbol{#1}}
\newcommand{\mcl}[1]{\mathcal{#1}}

\newcommand{\NN}{\mathbb{N}}
\newcommand{\rme}{\text{e}}
\newcommand{\rmi}{\text{i}}
\newcommand{\ud}{\,\textrm{d}}
\newcommand{\udd}{\textrm{d}}

\def\XXint#1#2#3{{\setbox0=\hbox{$#1{#2#3}{\int}$}
\vcenter{\hbox{$#2#3$}}\kern-.5\wd0}}

\newtheorem{thm}{Theorem}[section]

\newtheorem{lmm}[thm]{Lemma}

\theoremstyle{definition}
\newtheorem{dfn}[thm]{Definition}
\newtheorem*{ack}{Acknowledgments}
\theoremstyle{remark}
\newtheorem*{rem}{Remark}
\begin{document}


\title{Benjamini-Schramm convergence and limiting eigenvalue density of random matrices}

\date{}

\author{Sergio Andraus\\\footnotesize{Department of Physics, Chuo University, Tokyo 112-8551, Japan}}

\maketitle

\begin{abstract}
We review the application of the notion of local convergence on locally finite randomly rooted graphs, known as Benjamini-Schramm convergence, to the calculation of the global eigenvalue density of random matrices from the $\beta$-Gaussian and $\beta$-Laguerre ensembles. By regarding a random matrix as the weighted adjacency matrix of a graph, and choosing the root of such a graph with uniform probability, one can use the Benjamini-Schramm limit to produce the spectral measure of the adjacency operator of the limiting graph. We illustrate how the Wigner semicircle law and the Marchenko-Pastur law are obtained from this machinery.
\end{abstract}

\section{Introduction}\label{sec:intro}
The one-point density of the eigenvalues of random matrices from the Gaussian and Wishart ensembles, given by the Wigner semicircle and Marchenko-Pastur laws respectively, are well-known objects in random matrix theory, and they have been derived by several different methods (see, e.g., \cite{mehta,forrester}), most notably the orthogonal polynomial method and the method of eigenfunction expansions in the Coulomb gas analogy. These methods rely on the direct calculation of eigenvalue densities for matrices of finite size, after which the infinite size limit is taken and the (one-point) eigenvalue density is recovered from the dominant-order quantities.

In this review, we focus on a different approach: we illustrate how to use the Benjamini-Schramm convergence on randomly rooted locally finite graphs \cite{benjaminischramm} to obtain an object which contains the information of the eigenvalue density of the random matrix ensemble in question \textit{after} taking the infinite-size limit. The Benjamini-Schramm convergence was initially developed with the goal of proving that, if one considers a random walk on a finite graph which has a randomly rooted locally finite limiting graph, the random walk on the limiting graph is recurrent. However, this notion of convergence can be used to study the behavior of other quantities related to the limiting graph, such as its adjacency operator and its eigenvalues.

The connection between the Benjamini-Schramm convergence and random matrix theory comes from regarding any one particular ensemble of random matrices as a set of adjacency operators on graphs. In particular, the $\beta$-ensembles introduced by Dumitriu and Edelman \cite{dumitriuedelman} and Edelman and Sutton \cite{edelmansutton}, which extend the classical (threefold) random matrix ensembles from the discrete parameter $\beta=$ 1, 2 or 4 to $\beta$ real and positive, are sparse matrices with finite entries almost surely. These two properties (sparsity and boundedness of entries) turn out to be critical in the application of the method reviewed here. Once one considers the graph represented by the matrix ensemble in question, one can take the Benjamini-Schramm limit to obtain the randomly rooted limiting graph and calculate the spectral measure of its adjacency operator, and subsequently the eigenvalue density of the random matrix ensemble in question by using the results in \cite{abertthomvirag,ViragICM2014}.

We review the definition of the Benjamini-Schramm convergence in Sec.~2, we study the adjacency operator and its spectral measure in Sec.~3, and we illustrate the cases of the $\beta$-Hermite and $\beta$-Laguerre ensembles in Sec.~4. We offer a few concluding remarks in Sec.~5. This review is based on notes taken during lectures given by B. Vir{\'a}g at the Les Houches Physics School during the July 2015 summer school: Stochastic Processes and Random Matrices.

\section{The Benjamini-Schramm convergence}\label{sec:BSconvergence}

Following \cite{benjaminischramm}, we consider the set of connected graphs $G=(V,E)$, and we define rooted graphs as ordered pairs $(G,o)$ where the vertex $o\in V$ is the root. We define the space of isomorphism classes of rooted, connected, and locally finite graphs of maximum degree $D$ by $\mathcal{RG}_D$. This means that every vertex $v\in V$ of a rooted graph $(G,o)\in\mathcal{RG}_D$ has a finite number of neighbors, but the number of vertices in the graph may be infinite. Consider the locally finite rooted graphs $(G,o)$ and $(G^\prime,o^\prime)$. Then, we can define the metric
\begin{equation}\label{distancebtwgraphs}
d[(G,o),(G^\prime,o^\prime)]:=2^{-k[(G,o),(G^\prime,o^\prime)]},
\end{equation}
where
\begin{equation}
k[(G,o),(G^\prime,o^\prime)]:=\sup\{r\in\NN_0:B_r(G,o)\simeq B_r(G^\prime,o^\prime)\}
\end{equation}
and $B_r(G,o)$ is the subgraph of radius $r$ around the root $o$ of $G$. If the two rooted graphs are isomorphic, then $k$ tends to infinity and $d[(G,o),(G^\prime,o^\prime)]=0$. We see that $\mathcal{RG}_D$ is compact under the metric $d[(G,o),(G^\prime,o^\prime)]$.

We consider the random rooted graphs obtained by choosing the root $o$ with uniform probability from the vertices of $G$. Thanks to the metric in Eq.~\eqref{distancebtwgraphs} and the compactness of $\mathcal{RG}_D$, we can define probability measures on $\mathcal{RG}_D$. Consider a Borel set $\mathcal{B}\subset\mathcal{RG}_D$. Then, for the random graphs we consider, we denote the probability that $(G,o)\in\mathcal{B}$ with $o\in V$ chosen randomly uniformly by $\nu_G(\mcl{B})$. Then, $\nu_G(\{(G,u)\})=\nu_G(\{(G,v)\})$ for all $u,v\in V$ by definition. Then, in the case where $G$ is a finite graph, and denoting by $\sharp V$ the number of vertices in $G$,
\begin{equation}
\nu_G(\{(G,v)\})=\frac{1}{\sharp V}
\end{equation}
for all $v\in V$. When $G$ is infinite, its vertices are labeled with a continuous parameter $r\in[0,1]$, say, and $\nu_G(\{(G,v_r)_{r\in R}\})$ is the Lebesgue measure of the set $R\subset [0,1]$. If we write $\mcl{B}=\{(H_j,v_j)\}_j$, $\nu_G(\mcl{B})$ is given by
\begin{equation}
\nu_G(\mcl{B})=\mbs{1}[G\in\{H_j\}_j]\int_{r:(G,v_r)\in\mcl{B}}\nu_G(\{(G,v_r)\})\ud r,
\end{equation}
where $\mbs{1}[\cdot]$ is the indicator function. The probability that $(G,o)\in\mcl{B}$ is zero when no graph in $\mcl{B}$ is isomorphic to $G$, and the integral represents the fraction of all the random rooted graphs obtained from $G$ which lie in $\mcl{B}$. We denote by $\mbb{P}_G$ the probability law with respect to $\nu_G(\mcl{B})$, and its expectation will be denoted by $\mbb{E}_G$. We are now ready to give the definition of the Benjamini-Schramm convergence.

\begin{dfn}[Benjamini-Schramm convergence]\label{bsconvergence}
Consider the sequence of rooted graphs $\{(G_n,o_n)\}_{n=0}^\infty$, with roots chosen randomly with uniform probability. The rooted random graph $(G,o)$ is the \textbf{distributional limit} of the sequence if for every $r > 0$ and every finite rooted graph $(H,o^\prime)$,
\begin{equation*}
\mathbb{P}_{G_n}[(H,o^\prime)\simeq (B_r(G_n,o_n),o_n)]\stackrel{j\to\infty}{\longrightarrow} \mathbb{P}_{G}[(H,o^\prime)\simeq (B_r(G,o),o)],
\end{equation*}
that is, if the law of $(G_n,o_n)$ tends locally weakly to the law of $(G,o)$ as $j\to\infty$.
\end{dfn}

\section{The adjacency operator and its spectral measure}

Let us continue by considering the adjacency operator $A$ of a locally finite graph $G=(V,E)$. This is an operator defined on the space $\mathscr{L}^2(G)$ of complex, square-summable functions $f$ on the vertex set $V$, for which we define the inner product
\begin{equation}
\langle f,g\rangle=\sum_{v\in V}\bar{f}(v)g(v).
\end{equation}
That is, $A: \mathscr{L}^2(G)\to\mathscr{L}^2(G)$, and its action on the function $f\in\mathscr{L}^2(G)$ is
\begin{equation}
[Af](v)=\sum_{u:(v,u)\in E} l((v,u))f(u),
\end{equation}
where $l((v,u))$ denotes the weight of the edge $(v,u)$ connecting the vertices $v$ and $u$. From our previous assumptions, the number of edges connected to any one vertex with nonzero weight is bounded. In addition, we require that the weights be symmetric, i.e., $l((v,u))=l((u,v))$, real and bounded in absolute value. We denote the bound on the weights by $M_w$. Under these conditions, it follows that the adjacency operator is bounded, and by the spectral theorem there exists an orthonormal basis of $\mathscr{L}^2(G)$, which we denote by $\{e_r\}_r$, such that
\begin{equation}
Ae_r(v)=\lambda_re_r(v)
\end{equation}
and with $\lambda_r\in \mbb{R}$. 

Note that the adjacency operator only depends on the graph $G$, and does not depend on the choice of a root $o$. However, the Benjamini-Schramm limit is taken with respect to the measure of a limiting random rooted graph, and for that reason we introduce the spectral measure of the adjacency operator with respect to the root. Denote the projection operator-valued measure on Borel sets $X\subset \mbb{R}$ acting on $f\in\mathscr{L}^2(G)$ by $[P_Xf](v)$; the characteristic function $\chi_o(v)$ is equal to one if $o=v\in V$ and zero otherwise.
\begin{dfn}\label{spectralmeasure}
The spectral measure of $A$ with respect to the root $o$ is defined as
\begin{equation}
\mu_{G,o}(X):=\sum_{v\in V}\overline{[P_X\chi_o]}(v)\chi_o(v)=\langle P_X\chi_o,\chi_o\rangle.
\end{equation}
\end{dfn}
One can recover the spectral measure of $A$ from this expression if $G_n$ is a finite graph with $n$ vertices. Indeed, $\mu_{G_n,o}(X)$ is given by
\begin{eqnarray}
\sum_{v\in V}\overline{[P_X\chi_o]}(v)\chi_o(v)&=&\sum_{v\in V}\sum_{u\in V}\sum_{m=1}^n\mbs{1}[\lambda_m\in X]\bar{e}_m(u)\chi_o(u)e_m(v)\chi_o(v)\\
&=&\sum_{m=1}^n\mbs{1}[\lambda_m\in X]|e_m(o)|^2,
\end{eqnarray}
and the expected measure $\mu_{G_n}$, which we define as the expectation with respect to $G_n$ of the spectral measure at $o$, is given by
\begin{equation}\label{finitespectralmeasure}
\mu_{G_n}(X):=\mbb{E}_{G_n}[\mu_{G_n,o}(X)]=\sum_{o\in V}\mu_{G_n,o}(X)\nu_{G_n}(\{(G_n,o)\})=\frac{1}{n}\sum_{m=1}^n\mbs{1}[\lambda_m\in X]
\end{equation}
due to the orthonormality of the basis $\{e_m\}_{m=1}^n$ and because $\nu_{G_n}(\{(G_n,o)\})=1/n$. This is the spectral measure of $A$.

When $G$ is infinite, however, the last expression in Eq.~\eqref{finitespectralmeasure} may not be well-defined. Therefore, Def.~\ref{spectralmeasure} is useful in making sense of the spectral measure of $A$ when $G$ is infinite. In fact, if $G$ is the limit of a sequence of finite random rooted graphs (in this case, we say that $G$ is sofic), the expected measure of $\mu_{G,o}(X)$ is the spectral measure of $A$. Ab{\'e}rt, Thom and Vir{\'a}g \cite{abertthomvirag} have proved the pointwise convergence of the expected measure, and we give its statement as follows.
\begin{lmm}\label{pointwiseconvergence}
Let $(G,o)$ be a sofic random rooted graph, and let $\{(G_n,o_n)\}_{n=1}^\infty$ be a sequence of finite random rooted graphs converging to $(G,o)$. Then, $\mu_{G_n}(\{x\})$ converges to $\mu_{G}(\{x\})$ for every $x\in\mbb{R}$.
\end{lmm}
We omit the proof for brevity. We finish this section by listing the steps for finding the spectral measure of the adjacency operator of the limiting graph:
\begin{enumerate}
\item
Start with a sequence of Benjamini-Schramm converging graphs with adjacency operators that satisfy the requirements of symmetry, bounded degree and bounded edge weights.
\item
Calculate the eigenfunctions and eigenvalues of the limiting adjacency operator with respect to the random root.
\item
Calculate the spectral measure with respect to the random root.
\item
Take the expectation over the roots to obtain the spectral density of the limiting adjacency operator.
\end{enumerate}

\section{The $\beta$-Hermite and $\beta$-Laguerre cases}

Let us now consider random matrices of size $n\times n$ from the $\beta$-Hermite (or Gaussian) and $\beta$-Laguerre ensembles \cite{dumitriuedelman}. These are tridiagonal matrices given, in the $\beta$-Hermite case, by the random entries
\begin{equation}
\begin{array}{ccl}
a_j^\text{H}&\sim&\mcl{N}(0,2)/\sqrt{n},\\
b_j^\text{H}&\sim&\chi_{\beta j}/\sqrt{n},
\end{array}\quad 1\leq j\leq n,
\end{equation}
and a matrix $H_\beta^{(n)}$ from the ensemble is given by
\begin{equation}
H_\beta^{(n)}=\left(\begin{array}{cccc}
a_1^\text{H}&b_1^\text{H}&&\\
b_1^\text{H}&a_2^\text{H}&\ddots&\\
&\ddots&\ddots&b_{n-1}^\text{H}\\
&&b_{n-1}^\text{H}&a_n^\text{H}
\end{array}\right)=:\text{Tridiag}_n(\{a_j^\text{H}\}_{j=1}^n,\{b_j^\text{H}\}_{j=1}^{n-1}).
\end{equation}
In the case of the $\beta$-Laguerre ensemble, given the parameters $\gamma\geq 1$ and $\alpha:=\beta\gamma(n-1)/2$, a matrix $L_\beta^{(n)}$ from the ensemble is given by
\begin{equation}
L_\beta^{(n)}=\text{Tridiag}_n(\{a_j^\text{L}\}_{j=0}^{n-1},\{b_j^\text{L}\}_{j=1}^{n-1}),\ \begin{array}{ccl}
a_0^\text{L}&\sim&\chi^2_{2\alpha}/n,\\
a_j^\text{L}&\sim&\chi^2_{2\alpha+\beta(n-2j)}/n,\\
b_j^\text{L}&\sim&\chi_{2\alpha-\beta(j-1)}\chi_{\beta(n-j)}/n.
\end{array}
\end{equation}

Because they are tridiagonal, these matrices represent a graph in which each vertex is connected to two neighbors through edges with random weight $b_j$, and to itself through an edge with weight $a_j$, as depicted in Fig.~\ref{fig:beforebslimit}. 
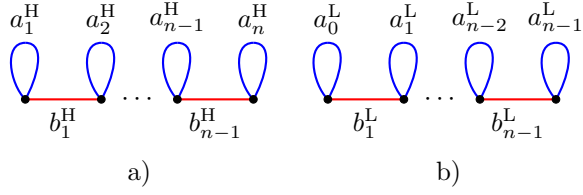
\begin{figure}
\begin{equation*}
\begin{array}{cc}
\begin{tikzpicture}
\node [above] at (-1.5,0.75) {$a_1^\text{H}$};
\draw [-, blue, thick] (-1.5,0) to [out=60,in=0] (-1.5,0.75);
\draw [-, blue, thick] (-1.5,0.75) to [out=180,in=135] (-1.5,0);
\draw [-, red, thick] (-1.5,0) -- (-0.5,0);
\draw [fill] (-1.5,0) circle [radius=0.05]; 
\node [below] at (-1,0) {$b_1^\text{H}$};
\node [above] at (-0.5,0.75) {$a_2^\text{H}$};
\draw [-, blue, thick] (-0.5,0) to [out=60,in=0] (-0.5,0.75);
\draw [-, blue, thick] (-0.5,0.75) to [out=180,in=135] (-0.5,0);
\draw [fill] (-0.5,0) circle [radius=0.05]; 
\node at (0,0) {$\cdots$};
\node [above] at (0.5,0.75) {$a_{n-1}^\text{H}$};
\draw [-, blue, thick] (0.5,0) to [out=60,in=0] (0.5,0.75);
\draw [-, blue, thick] (0.5,0.75) to [out=180,in=135] (0.5,0);
\draw [-, red, thick] (0.5,0) -- (1.5,0);
\draw [fill] (0.5,0) circle [radius=0.05]; 
\node [below] at (1,0) {$b_{n-1}^\text{H}$};
\node [above] at (1.5,0.75) {$a_{n}^\text{H}$};
\draw [-, blue, thick] (1.5,0) to [out=60,in=0] (1.5,0.75);
\draw [-, blue, thick] (1.5,0.75) to [out=180,in=135] (1.5,0);
\draw [fill] (1.5,0) circle [radius=0.05];
\end{tikzpicture}&
\begin{tikzpicture}
\node [above] at (-1.5,0.75) {$a_0^\text{L}$};
\draw [-, blue, thick] (-1.5,0) to [out=60,in=0] (-1.5,0.75);
\draw [-, blue, thick] (-1.5,0.75) to [out=180,in=135] (-1.5,0);
\draw [-, red, thick] (-1.5,0) -- (-0.5,0);
\draw [fill] (-1.5,0) circle [radius=0.05]; 
\node [below] at (-1,0) {$b_1^\text{L}$};
\node [above] at (-0.5,0.75) {$a_1^\text{L}$};
\draw [-, blue, thick] (-0.5,0) to [out=60,in=0] (-0.5,0.75);
\draw [-, blue, thick] (-0.5,0.75) to [out=180,in=135] (-0.5,0);
\draw [fill] (-0.5,0) circle [radius=0.05]; 
\node at (0,0) {$\cdots$};
\node [above] at (0.5,0.75) {$a_{n-2}^\text{L}$};
\draw [-, blue, thick] (0.5,0) to [out=60,in=0] (0.5,0.75);
\draw [-, blue, thick] (0.5,0.75) to [out=180,in=135] (0.5,0);
\draw [-, red, thick] (0.5,0) -- (1.5,0);
\draw [fill] (0.5,0) circle [radius=0.05]; 
\node [below] at (1,0) {$b_{n-1}^\text{L}$};
\node [above] at (1.5,0.75) {$a_{n-1}^\text{L}$};
\draw [-, blue, thick] (1.5,0) to [out=60,in=0] (1.5,0.75);
\draw [-, blue, thick] (1.5,0.75) to [out=180,in=135] (1.5,0);
\draw [fill] (1.5,0) circle [radius=0.05];
\end{tikzpicture}\\
\text{a)}&\text{b)}
\end{array}
\end{equation*}
\caption{\label{fig:beforebslimit}Graphs corresponding a) to the $\beta$-Hermite and b) to the $\beta$-Laguerre ensembles before taking the Benjamini-Schramm limit.}
\end{figure}

We present the result of using the procedure outlined in the previous section on these ensembles in the following two theorems.
\begin{thm}\label{hermitecase}
The sequence of random rooted graphs obtained from the adjacency matrix $\{H_\beta^{(n)}\}_{n=1}^\infty$ is Benjamini-Schramm convergent in the limit $n\to\infty$, and the expected measure of the limiting adjacency operator is given by
\begin{equation}\label{wignersemicircle}
\mu_{H}(\udd x)=\mbs{1}[x\in[-2\sqrt{\beta},2\sqrt{\beta}]]\frac{\sqrt{4\beta-x^2}}{2\pi\beta}\ud x.
\end{equation}
\end{thm}

For the $\beta$-Laguerre ensemble, the limiting measure depends on the parameter $\gamma$ in the form of the quantities $L_\pm:=\beta(1\pm \sqrt{\gamma})^2$.
\begin{thm}\label{laguerrecase}
The sequence of random rooted graphs obtained from the adjacency matrix $\{L_\beta^{(n)}\}_{n=1}^\infty$ is Benjamini-Schramm convergent in the limit $n\to\infty$, and the expected measure of the limiting adjacency operator is given by
\begin{equation}\label{marchenkopastur}
\mu_L(\udd x)=\mbs{1}[x\in[L_-,L_+]]\frac{\sqrt{(x-L_-)(L_+-x)}}{2\pi \beta x}\ud x.
\end{equation}
\end{thm}

We present the proof of both statements in succession.
\begin{proof}[Proof of Thm.~\ref{hermitecase}]
We take the Benjamini-Schramm limit as follows. Denote the finite graph in Fig.~\ref{fig:beforebslimit}a) by $H_n$. The graph satisfies the assumptions of Def.~\ref{bsconvergence}. It suffices to show that the weights on the edges are finite as $n\to\infty$. Assume that we label the vertices with integers in $\{1,\ldots,n\}$ and that for every graph $H_n$ the root  is labeled $j_n$. Consider an integer sequence $\{j_n\}_{n=1}^\infty$ such that $j_n\to\infty$ and $j_n/n\to u\in[0,1]$. Then, choosing the root randomly uniformly is equivalent to setting $u\sim\text{Unif}[0,1]$. In the limit, we have
\begin{equation}
a_{j_n}^{H}\sim\frac{\mcl{N}(0,2)}{\sqrt{n}}\to 0,\quad b_{j_n}^{H}\sim\sqrt{\frac{\beta j_n}{n}}\frac{\chi_{\beta j_n}}{\sqrt{\beta j_n}}\stackrel{n\to\infty}{\longrightarrow}\sqrt{\beta u}
\end{equation}
almost surely as $n\to \infty$. The limit for $b_{j_n}^{H}$ follows from the properties of the moment generating function of the chi distribution. The limiting graph, which we denote by $H$, is depicted in Fig.~\ref{fig:afterbslimit}a).
\begin{figure}
\begin{equation*}
\begin{array}{cc}
\begin{tikzpicture}
\node at (-1,0) {$\cdots$};
\draw [-, red, thick] (-0.5,0) -- (0.5,0);
\draw [fill] (-0.5,0) circle [radius=0.05]; 
\node [below] at (0,0) {$\sqrt{\beta u}$};
\draw [fill] (0.5,0) circle [radius=0.05];
\node at (1,0) {$\cdots$};
\end{tikzpicture}&
\begin{tikzpicture}
\node at (-1,0) {$\cdots$};
\node [above left] at (-0.5,0.75) {$\beta(\gamma+1-2u)$};
\draw [-, blue, thick] (-0.5,0) to [out=60,in=0] (-0.5,0.75);
\draw [-, blue, thick] (-0.5,0.75) to [out=180,in=135] (-0.5,0);
\draw [-, red, thick] (-0.5,0) -- (0.5,0);
\draw [fill] (-0.5,0) circle [radius=0.05]; 
\node [below] at (0,0) {$\beta\sqrt{\gamma-u}\sqrt{1-u}$};
\node [above right] at (0.5,0.75) {$\beta(\gamma+1-2u)$};
\draw [-, blue, thick] (0.5,0) to [out=60,in=0] (0.5,0.75);
\draw [-, blue, thick] (0.5,0.75) to [out=180,in=135] (0.5,0);
\draw [fill] (0.5,0) circle [radius=0.05];
\node at (1,0) {$\cdots$};
\end{tikzpicture}\\
\text{a)}&\text{b)}
\end{array}
\end{equation*}
\caption{\label{fig:afterbslimit}Benjamini-Schramm limiting graphs for a) the $\beta$-Hermite and b) the $\beta$-Laguerre ensembles. }
\end{figure}
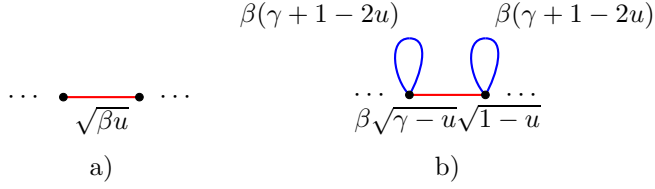
Note that $u$ indicates the section of the graph where the root was chosen, but the vertices are still labeled by integers. The action of the adjacency operator $A_H$ is given by
\begin{equation}
A_Hf(v)=\sqrt{\beta u}[f(v-1)+f(v+1)].
\end{equation}
This operator can be diagonalized by a Fourier basis, yielding the eigenvalues $\lambda_{u,\omega}^H$:
\begin{equation}\label{fourierbasis}
e_\omega(v)=\rme^{\rmi \omega v}/\sqrt{2\pi},\quad
\lambda_{u,\omega}^H=2\sqrt{\beta u}\cos(\omega).
\end{equation}
Here, $\rmi=\sqrt{-1}$ and $-\pi\leq \omega\leq\pi$. The next step is to calculate the spectral measure at $u$. From Def.~\ref{spectralmeasure} and Lemma~\ref{pointwiseconvergence}, and denoting the Dirac measure concentrated at $\lambda$ by $\delta_{\lambda}(X)$, $X\subset\mbb{R}$, we write
\begin{equation}
\mu_{H,u}(\udd x)=\int_{-\pi}^{\pi}\frac{1}{2\pi}\delta_{\lambda_{u,\omega}^H}(\udd x)\ud \omega=
\frac{\mbs{1} [x\in[-2\sqrt{\beta u},2\sqrt{\beta u}]]}{\pi \sqrt{4\beta u-x^2}}\ud x.
\end{equation}
Note that the measure is nonzero only when $u\geq x^2/(4 \beta)$ and that the singularities at $u= x^2/(4 \beta)$ pose no problem, because they are integrable. Finally, we take the expectation with respect to $u$. The result is
\begin{equation}
\mu_{H}(\udd x)=\int_0^1 \frac{\mbs{1} [x\in[-2\sqrt{\beta u},2\sqrt{\beta u}]]}{\pi \sqrt{4\beta u-x^2}} \ud u \ud x
=\int_{x^2/4\beta}^1\!\!\! \frac{\mbs{1} [x\in[-2\sqrt{\beta},2\sqrt{\beta}]]}{\pi \sqrt{4\beta u-x^2}}\ud u \ud x.
\end{equation}
Performing the integral yields the result.
\end{proof}
\begin{rem}
The measure $\mu_{H}(X)$ is the well-known Wigner semi-circle law. Note that $\beta$ is simply a scale factor; setting $y:=\sqrt{\beta}x$ in Eq.~\eqref{wignersemicircle} yields the semicircle law for $\beta=1$. This is evidence of the universality of the semicircle distribution.
\end{rem}

\begin{proof}[Proof of Thm.~\ref{laguerrecase}]
As in the previous proof, we denote the graph in Fig.~\ref{fig:beforebslimit}b) by $L_n$. We take the Benjamini-Schramm limit by choosing a root $j_n$ such that $j_n\to\infty$ and $j_n/n\to u\in[0,1]$ and set $u\sim\text{Unif}[0,1]$. Then, the weights on the edges converge to
\begin{equation}
a_{j_n}^L\sim \frac{2\alpha+\beta(n-2j_n)}{n}\frac{\chi^2_{2\alpha+\beta(n-2j_n)}}{2\alpha+\beta(n-2j_n)}\stackrel{n\to\infty}{\longrightarrow}\beta(\gamma+1-2u),
\end{equation}
and
\begin{eqnarray}
b_{j_n}^L&\sim& \sqrt{\frac{2\alpha-\beta(j_n-1)}{n}}\frac{\chi_{2\alpha-\beta(j_n-1)}}{\sqrt{2\alpha-\beta(j_n-1)}}\sqrt{\frac{\beta(n-j_n)}{n}}\frac{\chi_{\beta(n-j_n)}}{\sqrt{\beta(n-j_n)}}\nonumber\\
&\stackrel{n\to\infty}{\longrightarrow}&\beta\sqrt{\gamma-u}\sqrt{1-u}
\end{eqnarray}
almost surely, by the properties of the moment generating functions of the chi and chi-square distributions. The limiting graph, $L$, is shown in Fig.~\ref{fig:afterbslimit}b). The action of the adjacency operator $A_L$ is then given by
\begin{equation}
A_Lf(v)=\beta\sqrt{\gamma-u}\sqrt{1-u}[f(v-1)+f(v+1)]+\beta(\gamma+1-2u)f(v),
\end{equation}
and using the Fourier basis in Eq.~\eqref{fourierbasis} we find that the eigenvalues are given by
\begin{equation}
\lambda_{u,\omega}^L=c_1(u)+2c_2(u)\cos\omega,\ c_1(u)=\beta(\gamma+1-2u),\ c_2(u)=\beta\sqrt{\gamma-u}\sqrt{1-u}.
\end{equation}
The spectral measure at $u$ is given by
\begin{equation}
\mu_{L,u}(\udd x)=\frac{\mbs{1}[x-c_1(u)\in[-2c_2(u),2c_2(u)]]}{\sqrt{4c_2^2(u)-(x-c_1(u))^2}}\frac{\udd x}{\pi}.
\end{equation}
The argument in the indicator function comes from the fact that the measure must be zero if $x$ is not in the image of $\lambda_{u,\omega}^L$ for $-\pi\leq\omega\leq\pi$, that is, $x$ must be in the interval $[c_1(u)-2c_2(u),c_1(u)+2c_2(u)]$. This is equivalent to requiring that
\begin{equation}
u\leq [\beta(1+\sqrt{\gamma})^2/x-1][x/\beta-(1-\sqrt{\gamma})^2]=\frac{(L_+-x)(x-L_-)}{\beta x}=:l_\beta(x).
\end{equation}
Because $u\in[0,1]$, we must require that $l_\beta(x)$ be positive, and because $L_+>L_-\geq 0$ for $\gamma\geq 1$ we see that the measure is positive for $(L_+-x)(x-L_-)> 0$. Then, we write
\begin{equation}
\mu_L(\udd x)=\mbb{E}_L[\mu_{L,u}(\udd x)]=\int_0^{l_\beta(x)}\!\!\!\frac{\mbs{1}[x\in[L_-,L_+]}{\sqrt{2\beta x(\gamma+1)-\beta^2(\gamma-1)^2-4\beta x u -x^2}}\frac{\udd u}{\pi}\ud x.
\end{equation} 
By computing the integral, the claim is proved.
\end{proof}
\begin{rem}
In this case, we obtain the Marchenko-Pastur law. As before, $\beta$ is simply a scale factor, which is evidence of the universality of this distribution. This means that $\gamma$ dictates the shape of the distribution. Also, note that if $\beta=1$, the case we consider here, $\gamma\geq 1$, corresponds to the matrices from the Wishart ensembles given by $L_1=B_1 B_1^T$, where $B_1$ is a real, rectangular matrix with Gaussian-distributed entries and dimensions $n\times m$ with $m\geq n$. In other words, $L_\beta$ does not have a concentrated density of eigenvalues at zero almost surely. The case where $0<\gamma<1$ can be treated using the method presented here, but care must be taken in calculating the tridiagonal form. 
\end{rem}

\section{Concluding remarks}

Similar results to those illustrated here can be found for the $\beta$-Jacobi ensembles. If the matrices in question are sparse (i.e., the number of nonzero entries per row is bounded) and its entries themselves are bounded, the method shown here should be applicable. However, this requirement makes the use of the Benjamini-Schramm limit ineffective in treating problems such as finding the spectral measure in $\mbb C$ of the Ginibre ensemble, as it cannot be reduced into a manageable sparse matrix ensemble. The method itself is interesting, however, and we plan to find other applications for it in the future, such as the time evolution of the spectral measure of sparse matrix-valued stochastic processes.

\begin{ack}
The author would like to thank the organizers of the Probability Theory Symposium 2016 held at RIMS, Kyoto University, on Dec. 19-22 2016, where this work was presented. The author would also like to thank the organizers of the summer school on Random Matrices and Stochastic Processes at the Les Houches Physics School (July 2015), where this work was carried out in part, and B. Vir{\'a}g, for his enlightening lectures. Finally, the author would like to thank M. Katori for his careful reading of this manuscript. This work was supported in part by the Grant-in-Aid for Scientific Research (B) (No. 26287019) of the Japan Society for the Promotion of Science.
\end{ack}


\footnotesize{E-mail address: andraus@phys.chuo-u.ac.jp}

\begin{thebibliography}{99}
\bibitem{mehta}
Mehta, M. L.,
\textit{Random Matrices},
3rd ed.,
Elsevier, 2004.

\bibitem{forrester}
Forrester, P. J.,
\textit{Log-Gases and Random Matrices},
Princeton University Press, 2010.

\bibitem{benjaminischramm}
Benjamini, I, Schramm, O., 
Recurrence of distributional limits of finite planar graphs,
\textit{Elec. J. Prob.}
\textbf{6} (2001) 23, 1-13.

\bibitem{dumitriuedelman} 
Dumitriu, I., Edelman A., 
Matrix models for beta-ensembles,
\textit{J. Math. Phys.}, 
\textbf{43} (2002) 11, 5830-5847.

\bibitem{edelmansutton}
Edelman, A., Sutton, B. D.,
The Beta-Jacobi Model, the CS Decomposition, and Generalized Singular Value Problems,
\textit{Found. Comput. Math.},
\textbf{8} (2008) 2, 259-285.

\bibitem{abertthomvirag} 
Ab{\'e}rt, M., Thom, A., Vir{\'a}g, B., 
Benjamini - Schramm convergence and pointwise convergence of the spectral measure,
{preprint:} 
http://www.renyi.hu/$\sim$abert/luckapprox.pdf.

\bibitem{ViragICM2014} 
Vir{\'a}g, B., 
Operator limits of random matrices,
\textit{Proceedings of the ICM Seoul},
\textbf{4} (2014) 247-271.

%
%
%
%
%
%
%
%
%

\end{thebibliography}
\end{document}